\documentclass[a4paper]{article}

\usepackage{amsmath}
\usepackage{amssymb}
\usepackage{array}
\usepackage[table]{xcolor}
\usepackage{breqn}
\usepackage{listings} 
\usepackage{amsthm}
\usepackage{mathrsfs}
\usepackage{enumerate} 
\usepackage{changepage}
\usepackage{tabularx}
\usepackage{pdflscape}
\usepackage{afterpage}
\usepackage{longtable}

\newtheorem{thm}{Theorem}[section]
\newtheorem{defn}[thm]{Definition}
\newtheorem{lem}[thm]{Lemma}
\newtheorem{prop}[thm]{Proposition}

\newtheorem*{rem}{Remark}

\newcommand{\ml}[2]{\begin{tabular}{@{} >{$}#1<{$} @{}} #2 \end{tabular}}

\usepackage[top=1in, bottom=1.25in, left=1.25in, right=1.25in]{geometry}

\makeatletter
\newcommand{\thickhline}{%
    \noalign {\ifnum 0=`}\fi \hrule height 1pt
    \futurelet \reserved@a \@xhline
}
\newcolumntype{"}{@{\hskip\tabcolsep\vrule width 1.25pt\hskip\tabcolsep}}
\makeatother

\usepackage{authblk}

\title{Majorana algebras generated by a $2A$ algebra and one further axis}
\author[1]{Madeleine L.~Whybrow}
\affil[1]{Department of Mathematics, Imperial College London, SW7 2AZ madeleine.whybrow10@imperial.ac.uk}
\date{}

\begin{document}

\maketitle

\begin{abstract}
We consider Majorana algebras generated by three Majorana axes $a_0$, $a_1$ and $a_2$ such that $a_0$ and $a_1$ generate a dihedral algebra of type $2A$. We show that such an algebra must occur as a Majorana representation of one of $27$ groups. These $27$ groups coincide with the subgroups of the Monster that are generated by three $2A$-involutions $a$, $b$ and $c$ such that $ab$ is also a $2A$-involution, which were classified by S.~P.~Norton in 1985. Our work relies on that of S.~Decelle and consists of showing that certain groups do not admit Majorana representations.  
\end{abstract}

\section{Introduction}

Majorana theory was introduced by Ivanov in \cite{Ivanov09} as an axiomatic framework in which to study objects related to the Monster simple group $\mathbb{M}$ and the Griess algebra $V_{\mathbb{M}}$. It is well known that the Monster group is a $6$-transposition group; it is generated by the $2A$ class of involutions and for all $t_0,t_1 \in 2A$, $o(t_0t_1) \leq 6$. Moreover, $t_0t_1$ must lie in one of the nine conjugacy classes $1A$, $2A$, $2B$, $3A$, $3C$, $4A$, $4B$, $5A$, $6A$. 

Conway \cite{Conway84} showed that there exists a bijection $\psi$ from the $2A$-involutions in $\mathbb{M}$ to idempotents in the Griess algebra, known as $2A$-axes. Moreover, the Griess algebra is generated by the $2A$-axes and for all $t_0,t_1 \in 2A$, the dihedral subalgebra $\langle \langle \psi(t_0), \psi(t_1) \rangle \rangle$ has one of nine possible isomorphism types, depending on the $\mathbb{M}$-conjugacy class of $t_0t_1$. 

The moonshine module $V^{\natural} = \bigoplus_{n = 0}^{\infty} V_n$ is a vertex operator algebra (VOA). It was constructed by Frenkel, Lepowsky and Meurman in \cite{FLM88} and was later used by Borcherds in his proof of Conway and Norton's Monstrous Moonshine conjectures \cite{Borcherds92}. The automorphism group of $V^{\natural}$ is the Monster group and the weight 2 subspace $V_2^{\natural}$ has the structure of a commutative algebra that coincides with the Griess algebra. 

In general, if we take $V = \bigoplus_{n = 0}^{\infty} V_n$ to be a VOA such that $V_0 = \mathbb{R}\mathbf{1}$ and $V_1 = 0$ then $V_2$ is known as a \emph{generalised Griess algebra} and has the structure of a commutative, non-associative real algebra. In particular, Miyamoto \cite{Miyamoto96} showed that there exist involutions $\tau_a \in \mathrm{Aut}(V)$ that are in bijection with certain vectors $a \in V_2$ known as \emph{Ising vectors}. Moreover, when $V = V^{\natural}$, the vectors $\frac{1}{2}a$ are the $2A$-axes of the Griess algebra and the $\tau_a$ are the $2A$-involutions of $\mathbb{M}$. Sakuma then proved the following result, now known as \emph{Sakuma's theorem}.
\begin{thm}[\cite{Sakuma07}]
\label{thm:sakuma}
If $V_2$ is a generalised Griess algebra and $a_1, a_2 \in V_2$ are Ising vectors then the subalgebra $\langle \langle a_1, a_2 \rangle \rangle$ is isomorphic to one of the nine dihedral subalgebras of the Griess algebra.
\end{thm}
This was a remarkable result that reproved the classification of the dihedral algebras of the Griess algebra, but in the more general setting of VOAs. Inspired by this result, Ivanov \cite{Ivanov09} introduced Majorana theory as an axiomatisation of certain properties of the $2A$-axes of the Griess algebra. The objects at the centre of the theory are \emph{Majorana algebras} which are generated by idempotents known as \emph{Majorana axes}. Majorana algebras are studied both in their own right and as \emph{Majorana representations} of certain groups. 

Given a group $G$, it is possible that it does not admit any Majorana representations. However, if $G$ is isomorphic to a subgroup of $\mathbb{M}$ that is generated by $2A$-involutions then it must admit at least one Majorana representation, which will be isomorphic to the group's corresponding subalgebra in the Griess algebra. 

It was shown in \cite{IPSS10} that a Majorana algebra generated by two Majorana axes must be isomorphic to a dihedral subalgebra of the Griess algebra (see Theorem \ref{thm:IPSS10}). Inspired by this result, we consider the question of classifying Majorana algebras generated by three Majorana axes $a_0$, $a_1$ and $a_2$ such that $a_0$ and $a_1$ generate a dihedral algebra of type $2A$. 

This problem was first posed at the end of Section 9.2 of \cite{Ivanov09}. In this paper we show that such an algebra must occur as a Majorana representation of one of $27$ groups, each of which is isomorphic to a subgroup of the Monster.

\section{Majorana theory}
Let $V$ be a real vector space equipped with an inner product $(\, , \, )$ and a bilinear, commutative, non-associative algebra product $\cdot$. Suppose that $(V,(\, , \, ), \cdot)$ obeys the following axioms:
\begin{description}
\item[M1] $( \, , \, )$  associates with $\cdot$ in the sense that 
\[
(u , v \cdot w ) = (u \cdot v, w ) 
\]
for all $u,v,w \in V$;
\item[M2] the Norton Inequality holds so that  
\[
(u \cdot u, v \cdot v) \geq (u \cdot v, u \cdot v)
\]
for all $u,v \in V$.
\end{description} 
Moreover, let $A$ be a subset of $V \backslash \{0\}$ and suppose that for every $a \in A$ the following conditions ((M3) to (M7)) hold: 
\begin{description}
\item[M3] $(a,a) = 1$ and $a \cdot a = a$, so that the elements of $A$ are idempotents of length 1;
\item[M4] $V = V_1^{(a)} \oplus V_0^{(a)} \oplus V_{\frac{1}{2^2}}^{(a)} \oplus V_{\frac{1}{2^5}}^{(a)}$ where $V_{\mu}^{(a)} = \{ v \mid v \in V, \, a \cdot v = \mu v\}$ is the set of $\mu$-eigenvectors of the adjoint action of $a$ on $V$;
\item[M5] $V_1^{(a)} = \{ \lambda a \mid \lambda \in \mathbb{R} \}$;
\item[M6] the linear transformation $\tau(a)$ of $V$ defined via 
\[
\tau(a): u \mapsto (-1)^{2^5 \mu}u
\]
for $u \in V_{\mu}^{(a)}$ with $\mu = 1,0,\frac{1}{2^2}, \frac{1}{2^5}$, preserves the algebra product (i.e. $u^{\tau(a)} \cdot v^{\tau(a)} = (u \cdot v)^{\tau(a)}$ for all $u, v \in V$);
\item[M7] if $V_+^{(a)}$ is the centraliser of $\tau(a)$ in $V$, so that $V_+^{(a)} = V_1^{(a)} \oplus V_0^{(a)} \oplus V_{\frac{1}{2^2}}^{(a)}$, then the linear transformation $\sigma(a)$ of $V_+^{(a)}$ defined via 
\[
\sigma(a) : u \mapsto (-1)^{2^2 \mu}u
\]  
for $u \in V_{\mu}^{(a)}$ with $\mu = 1,0,\frac{1}{2^2}$, preserves the restriction of the algebra product to $V_+^{(a)}$ (i.e. $u^{\sigma(a)} \cdot v^{\sigma(a)} = (u \cdot v)^{\sigma(a)}$ for all $u,v \in V_+^{(a)})$. 
\end{description}
\begin{defn}
\label{defn:majorana}
The elements of $A$ are called \emph{Majorana axes} while the automorphisms $\tau(a)$ are called \emph{Majorana involutions}. A real commutative non-associative algebra $(V, \cdot, (\, , \, ))$ is called a \emph{Majorana algebra} if it satisfies axioms (M1) and (M2) and is generated by a set of Majorana axes.
\end{defn}
\begin{defn}
If $V$ is a Majorana algebra generated by Majorana axes $A$ then for a subset $B \subseteq A$, we take $\langle \langle B \rangle \rangle$ to be the smallest subalgebra $U$ of $V$ that contains the Majorana axes $B$. 
\end{defn}
The axioms $(M1)$ to $(M7)$ imply that the eigenspaces $V_{\mu}^{(a)}$ of (the adjoint action of) $a$ satisfy the fusion rules in Table \ref{tab:fusion} as described explicitly below. 
\begin{lem}
\label{lem:fusion}
For a fixed Majorana axis $a$, if $u \in V_{\mu}^{(a)}$, $v \in V_{\nu}^{(a)}$ then the product $u \cdot v$ lies in the sum of eigenspaces with corresponding eigenvalues given by the \mbox{$(\mu,\nu)$-th} entry of Table \ref{tab:fusion}.
\end{lem}
\begin{proof}
As in the hypothesis, let $u \in V_{\mu}^{(a)}$, $v \in V_{\nu}^{(a)}$. If $1 \in \{\mu,\nu\}$ then the result follows from axiom (M4). If either zero or two of $\{\mu, \nu\}$ are equal to $\frac{1}{2^5}$ then $\tau(a)$ preserves $u \cdot v$ which must then lie in $V^{(a)}_1 \oplus V^{(a)}_0 \oplus V^{(a)}_{\frac{1}{2^2}}$. However, if exactly one of $\mu$ and $\nu$ is equal to $\frac{1}{2^5}$, then $\tau(a)$ inverts $u \cdot v$ and so $u \cdot v \in V^{(a)}_{\frac{1}{2^5}}$. 

A similar analysis using the action of $\sigma(a)$ shows that if $\{\mu, \nu\} = \{0, \frac{1}{2^2}\}$ then $u \cdot v \in V^{(a)}_{\frac{1}{2^2}}$  and if $\mu = \nu$ is equal to $0$ or $\frac{1}{2^2}$ then $u \cdot v \in V^{(a)}_1 \oplus V^{(a)}_0$. Finally, axioms (M1) and (M4) show that if $\mu = \nu = 0$ then the projection of $u \cdot v$ on to $V^{(a)}_1$ is zero. 
\end{proof}
\begin{table}
\centering
\begin{tabular}{>{$}c<{$}|>{$}c<{$}>{$}c<{$}>{$}c<{$}>{$}c<{$}} 
  & 1 & 0 & \frac{1}{2^2} & \frac{1}{2^5} \\ \hline
1 & 1 & \emptyset & \frac{1}{2^2} & \frac{1}{2^5} \\
0 & \emptyset & 0 & \frac{1}{2^2} & \frac{1}{2^5} \\
\frac{1}{2^2} & \frac{1}{2^2} & \frac{1}{2^2} & 1,0 & \frac{1}{2^5} \\
\frac{1}{2^5} & \frac{1}{2^5} & \frac{1}{2^5} & \frac{1}{2^5} & 1,0,\frac{1}{2^2} \\
\end{tabular}
\caption{The fusion rules}
\label{tab:fusion}
\end{table}
The Majorana axes and Majorana involutions correspond to the $2A$-axes in the Griess algebra and the $2A$-involutions in the Monster respectively. The fact that the $2A$-axes obey the axioms (M3) - (M7) was implicitly stated in \cite{Norton96} and was explicitly shown in Proposition 8.6.2 of \cite{Ivanov09}. The following is a natural axiomatisation of the relationship between the Monster and the Griess algebra.
\begin{defn}
A \emph{Majorana representation} is a tuple
\[
\textbf{R} = (G,T,V,\cdot, (,), \varphi, \psi)
\]
where 
\begin{itemize}
	\item $G$ is a finite group;
	\item $T$ is a $G$-stable set of generating involutions of $G$;
	\item $V$ is a real vector space equipped with an inner product $(\, , \, )$ and a commutative, bilinear product $\cdot$ satisfying (M1) and (M2) that is generated by a set $A$ of Majorana axes;
	\item $\varphi: G \rightarrow GL(V)$ is a linear representation that preserves both products,
	\item $\psi: T \rightarrow A$ is a bijective mapping such that
	\[
	\psi(t^g) = \psi(t)^{\varphi(g)}.
	\] 
\end{itemize}
\end{defn}
We can now state the seminal result of the theory, known as Sakuma's theorem. 
\begin{thm}[{{\cite{IPSS10}}}] 
\label{thm:IPSS10}
Let $\textbf{R} = (G,T,V,\cdot, (,), \varphi, \psi)$ be a Majorana representation of $G$, as defined above. For $t_0, t_1 \in T$ let $a_0 = \psi(t_0)$, $a_1 = \psi(t_1)$, $\tau_0 = \varphi(t_0)$, $\tau_1 = \varphi(t_1)$ and $\rho = t_0t_1$. Finally, let $D \leq GL(V)$ be the dihedral group $\langle \tau_0, \tau_1 \rangle$. Then 
\begin{enumerate}[(i)]
\item $|D| = 2N$ for $1 \leq N \leq 6$;
\item the subalgebra $U = \langle \langle a_0, a_1 \rangle \rangle$ is isomorphic to a dihedral algebra of type $NX$ for $X \in \{A,B,C\}$, the structure of which is given in Table \ref{tab:sakuma};
\item for $i \in \mathbb{Z}$ and $\epsilon \in \{0,1\}$, the vector $a_{2i+\epsilon}$ is the image of $a_{\epsilon}$ under the $i$-th power of $\rho$ and $\tau(a_{2i + \epsilon}) = \rho^{-i}\tau_{\epsilon}\rho^i$. 
\end{enumerate}
\end{thm}
Table \ref{tab:sakuma} does not show all pairwise algebra and inner products of the basis vectors. Those that are missing can be recovered from the action of the group $\langle \tau_0, \tau_1 \rangle$ together with the symmetry between $a_0$ and $a_1$. We also note that the dihedral algebra of type $1A$ is a $1$-dimensional algebra generated by one Majorana axis and so is omitted from Table \ref{tab:sakuma}.

The following lemma can be deduced from the structure of the dihedral algebras (see Lemma 2.20, \cite{IPSS10}).
\begin{lem}
\label{lem:inclusion}
Let $U$ be an algebra of type $NX$ (as in Table \ref{tab:sakuma}) that is generated by Majorana axes $a_0$ and $a_1$. Then
\begin{enumerate}[(i)]
\item if $U$ is of type $4A$, $4B$ or $6A$ then the subalgebra generated by $a_0$ and $a_2$ is of type $2B$, $2A$ or $3A$ respectively;
\item if $U$ is of type $6A$ then the subalgebra generated by $a_0$ and $a_3$ is of type $2A$. 
\end{enumerate} 
\end{lem}
\begin{table}
\begin{center}
\vspace{0.35cm}
\noindent
\begin{tabular}{|c|c|c|}
\hline
&&\\
 Type & Basis & Products and angles \\
&&\\
\hline
&&\\
2A & $a_0,a_1,a_\rho$ & $a_0 \cdot a_1=\frac{1}{2^3}(a_0+a_1-a_\rho),~a_0 \cdot a_\rho=\frac{1}{2^3}(a_0+a_\rho-a_1)$ \\
&&$(a_0,a_1)=(a_0,a_\rho)=(a_1,a_\rho)=\frac{1}{2^3}$\\
&& \\ 
2B & $a_0,a_1$ &$a_0 \cdot a_1=0$,~$(a_0,a_1)=0$ \\
&&\\
&  &$a_0 \cdot a_1=\frac{1}{2^5}(2a_0+2a_1+a_{-1})-\frac{3^3 \cdot 5}{2^{11}}u_\rho$\\
3A& $a_{-1},a_0,a_1,$ & $a_0 \cdot u_\rho=\frac{1}{3^2}(2a_0-a_1-a_{-1})+\frac{5}{2^5}u_\rho$~~~~\\
&$u_\rho$& $u_\rho \cdot u_\rho=u_\rho$\\
&& $(a_0,a_1)=\frac{13}{2^8}$,~$(a_0,u_\rho)=\frac{1}{2^2}$,~$(u_\rho,u_\rho)=\frac{2^3}{5}$
\\
&&\\
3C & $a_{-1},a_0,a_1$ & $a_0 \cdot a_1=\frac{1}{2^6}(a_0+a_1-a_{-1}),~(a_0,a_1)=\frac{1}{2^6}$ \\
&&\\ 
&  & ~$a_0 \cdot a_1=\frac{1}{2^6}(3a_0+3a_1+a_2+a_{-1}-3v_\rho)$\\
4A & $a_{-1},a_0,a_1,$ & $a_0 \cdot v_\rho=\frac{1}{2^4}(5a_0-2a_1-a_2-2a_{-1}+3v_\rho)$\\
&$a_2,v_\rho$&~$v_\rho \cdot v_\rho=v_\rho$, ~$a_0 \cdot a_2=0$ \\
& & $(a_0,a_1)=\frac{1}{2^5},~(a_0,a_2)=0,~(a_0,v_\rho)=\frac{3}{2^3},~(v_\rho,v_\rho)=2$\\
&&\\
4B & $a_{-1},a_0,a_1,$ & $a_0 \cdot a_1=\frac{1}{2^6}(a_0+a_1-a_{-1}-a_2+a_{\rho^2})$
\\
& $a_2,a_{\rho^2}$ & $a_0 \cdot a_2=\frac{1}{2^3}(a_0+a_2-a_{\rho^2})$\\
&& $(a_0,a_1)=\frac{1}{2^6},~(a_0,a_2)=(a_0,a_\rho)=\frac{1}{2^3}$ \\
&&\\
&& $a_0 \cdot a_1=\frac{1}{2^7}(3a_0+3a_1-a_2-a_{-1}-a_{-2})+w_\rho$
\\
 5A & $a_{-2},a_{-1},a_0,$ & $a_0 \cdot a_2=\frac{1}{2^7}(3a_0+3a_2-a_1-a_{-1}-a_{-2})-w_\rho$
\\
& $a_1,a_2,w_\rho$ & $a_0 \cdot w_\rho=\frac{7}{2^{12}}(a_{1}+a_{-1}-a_2-a_{-2})+\frac{7}{2^5}w_\rho$\\
& & $w_\rho \cdot w_\rho=\frac{5^2 \cdot 7}{2^{19}}(a_{-2}+a_{-1}+a_0+a_1+a_2)$\\
&&$(a_0,a_1)=\frac{3}{2^7},~(a_0,w_\rho)=0$, $(w_\rho,w_\rho)=\frac{5^3 \cdot 7}{2^{19}}$\\
&& \\
& & $a_0 \cdot a_1=\frac{1}{2^6}(a_0+a_1-a_{-2}-a_{-1}-a_2-a_3+a_{\rho^3})+\frac{3^2 \cdot 5}{2^{11}}u_{\rho^2}$\\
6A& $a_{-2},a_{-1},a_0,$ &$a_0 \cdot a_2=\frac{1}{2^5}(2a_0+2a_2+a_{-2})-\frac{3^3 \cdot 5}{2^{11}}u_{\rho^2}$  \\ 
&$a_1,a_2,a_3$  &$a_0 \cdot u_{\rho^2}=\frac{1}{3^2}(2a_0-a_2-a_{-2})+\frac{5}{2^5}u_{\rho^2}$  \\
&$a_{\rho^3},u_{\rho^2}$ & $a_0 \cdot a_3=\frac{1}{2^3}(a_0+a_3-a_{\rho^3})$, $a_{\rho^3} \cdot u_{\rho^2}=0$, $(a_{\rho^3},u_{\rho^2})=0$\\
&&$(a_0,a_1)=\frac{5}{2^8}$, $(a_0,a_2)=\frac{13}{2^8}$, $(a_0,a_3)=\frac{1}{2^3}$\\
&&\\
\hline
\end{tabular}
\caption{The dihedral Majorana algebras}
\label{tab:sakuma}  
\end{center}
\end{table}
If $t,s$ and $ts$ are distinct $2A$-involutions in the Monster, then any two of their corresponding $2A$-axes generate a dihedral algebra of type $2A$ which also contains the third of these axes. This property is independent of the axioms (M1) - (M7) and so we consider it as a further axiom, which we define as axiom (M8) below. We assume axiom (M8) for all Majorana representations henceforth, as do most (but not all) papers concerning Majorana theory.
\begin{description}
    \item[M8]  Let $t_i \in T$ and $a_i := \psi(t_i)$ for $ 0 \leq i \leq 2$. If $a_0$ and $a_1$ generate a dihedral subalgebra of type $2A$, then $t_0t_1 \in T$ and $\psi(t_0t_1) = a_{\rho}$. If $t_0t_1t_2 = 1$ then the subalgebra generated by $a_0$ and $a_1$ is of type $2A$ and $a_2 = a_{\rho}$.
\end{description}
Given $t_0, t_1 \in T$, it is usually possible to determine the isomorphism type of the algebra generated by $\psi(t_0)$ and $\psi(t_1)$ using Lemma \ref{lem:inclusion} and axiom (M8). For example, if $t_0, t_1 \in T$ such that $o(t_0t_1) = 6$ then $\langle \langle \psi(t_0), \psi(t_1) \rangle \rangle$ is of type $6A$ and $\langle \langle \psi(t_0), \psi(t_1t_0t_1t_0t_1) \rangle \rangle$ is of type $2A$ and so we must have $(t_0t_1)^3 \in T$. We rely heavily on this observation in Section \ref{sec:main}.

\section{Triangle-Point Groups}

With Sakuma's classification of the Majorana algebras generated by two Majorana axes complete, it is natural to consider the classification of algebras generated by three Majorana axes. Whilst the question of classifying all 3-generated Majorana algebras is clearly beyond the scope of this work (the Griess algebra itself is one such example), we consider a natural first step towards this question. 

In particular, we are interested in the classification of Majorana algebras generated by three axes, two of which generate a $2A$ dihedral algebra. We begin by considering the structure of the groups $G$ that may admit a Majorana representation $(G,T,V)$ where $V$ is of the desired form.

\begin{defn} Let $G$ be a group such that
\begin{enumerate}[(i)]
    \item $G$ is generated by three elements $a$, $b$, $c$ of order $2$ such that $ab$ is also of order $2$;
    \item for any two elements $t,s \in X := a^G \cup b^G \cup c^G \cup (ab)^G$, the product $ts$ has order at most 6;
\end{enumerate}  
then $G$ is a \emph{triangle-point} group. 
\end{defn}
\begin{prop}
Suppose that $V$ is a Majorana algebra generated by three Majorana axes $a_0,a_1,a_2$ such that the dihedral algebra $\langle \langle a_0,a_1 \rangle \rangle$ is of type $2A$. Then $V$ must exist as a Majorana representation $(G,T,V)$ where $G = \langle a,b,c \rangle $ is a triangle-point group and $T \subset G$ is such that $a,b,c,ab \in T$.  
\end{prop}
\begin{proof}
For $0 \leq i \leq 2$, let $\tau_i := \tau(a_i)$ and let $G := \langle \tau_0, \tau_1, \tau_2 \rangle$. We will chose a set $T$ of involutions of $G$ and define maps $\varphi$ and $\psi$ such that $(G,T,V,\varphi, \psi)$ is a Majorana representation. 

As $G$ is generated by Majorana involutions, we already have $G \leq GL(V)$ and so $\varphi$ is just the identity map. Moreover, if $a$ is a Majorana axis of $V$ such that $\tau(a) \in G$ and $g \in G$ then we define $\psi(\tau(a)^g) = a^{\varphi(g)}$.

We now choose $T \subset G$ to be a $G$-closed set of involutions of G such that $\tau_0, \tau_1, \tau_2 \in T$ and such that, if $t,s \in T$ and $o(ts) = 2$, then $ts \in T$ if and only if the algebra generated by $\psi(t)$ and $\psi(s)$ is of type $2A$. Then all elements of $T$ are of the form $\tau(a)^g$ for a Majorana axis $a \in V$ and $g \in G$ and so $\psi$ is a map from $T$ to $A$. Axiom (M8) and Theorem \ref{thm:IPSS10} imply that this is a bijection. 

Finally, from Theorem \ref{thm:IPSS10}, if $t,s \in T$ then $o(ts) \leq 6$ and so $G$ must be a triangle-point group, as required.
\end{proof}

These groups also provide further motivation to this question. We say that a triangle-point group $G = \langle a,b,c \rangle$, $2A$-\emph{embeds} into the Monster if there exists an injective homomorphism $f: G \rightarrow \mathbb{M}$ such that $f(a), f(b), f(c), f(ab) \in 2A$. The triangle-point subgroups which $2A$-embed into the Monster were studied by Norton in \cite{Norton85} in which he considered the possibility of constructing the Griess algebra using a permutation module constructed from triangle-point configurations in the Monster graph. 

In particular, he produced a list of the triangle-point groups that $2A$-embed into the Monster, although he did not include a proof of this result. Our work provides a proof that this list is complete (see the discussion following Theorem \ref{thm:main}). 

A crucial first step in the proof of our main result was completed by Decelle in Theorem 3.3 of \cite{Decelle13}.
\begin{thm}[\cite{Decelle13}]
\label{thm:sophie}
Each triangle-point group must occur as a quotient of at least one of the $11$ groups given in Table \ref{tab:tpgroups}. Each of these groups occurs as a quotient of a group of the form
\[
G^{(m,n,p)} := \langle a,b,c \mid a^2, b^2, c^2, (ab)^2, (ac)^m, (bc)^n, (abc)^p \rangle
\]
for $m,n,p \in [1..6]$, potentially with additional relations of the form $R_i^{r_i}$ for $i \in [1..5]$ and 
\[
R_1 := a \cdot b^c, \, R_2 := ab \cdot b^c, \, R_3 := ab \cdot a^c,  \, R_4 := c \cdot b^{ca}, \, R_5 := c^a\cdot c^{bc}. 
\]
\end{thm}
\begin{table}%
\begin{center}
\begin{tabular}{|>{$}c<{$}|>{$}c<{$} >{$}c<{$} >{$}c<{$} c|}
\hline
\textrm{Name} & \begin{tabular}{>{$}c<{$}} \textrm{Isomorphism} \\ \textrm{type} \end{tabular} & \begin{tabular}{>{$}c<{$}} G^{(m,n,p)} \\ (m,n,p) \end{tabular} & \begin{tabular}{>{$}c<{$}} \textrm{Added relations} \\ R_i^{r_i} \\ (r_1,r_2,r_3,r_4,r_5) \end{tabular} & $2A$-embeds in $\mathbb{M}$ \\ \hline
&&&&\\
G_1 & 2 \textrm{ wr } 2^2 & (4,4,4) & & Y  \\
G_2 & (S_3 \times S_3) : 2^2 & (4,4,6) & & Y  \\
G_3 & 2^4 : D_{10} & (4,5,5) & & Y  \\
G_4 & 2 \times S_5 & (4,5,6) & & Y  \\
G_5 & L_2(11) & (5,5,5) & & Y \\
G_6 & (2^4 : D_{12}) \times 2 & (4,6,6) & (4, -, -,-,-) & Y  \\ 
G_7 & 2^4 : A_5 & (5,5,6) & (-,5,-,-,-) & Y  \\
G_8 & 2 \times S_6 & (5,6,6) & (-,4,-,-,-) & Y  \\
G_9 & ( 2^4 : (S_3 \times S_3)) \times 2 & (6,6,6) & (4, \, 6, \, 6,-,-) & N \\
G_{10} & 2^5 : S_5 & (6,6,6) & (5, \, 5, \, 5, \, 4,-) &  Y  \\
G_{11} & (3^4 : 2) : (3^{1+2}_+ : 2^2) & (6,6,6) & (6, \, 6, \, 6,-,\,3) & N \\ &&&& \\ \hline
\end{tabular}
\caption{Triangle-Point Groups}
\label{tab:tpgroups}
\end{center}
\end{table}
\begin{rem}
It is important to note that any permutation of the elements $a$, $b$ and $ab$ in the presentation of $G := G^{(m,n,p)}$ induces an automorphism of $G$. Such an automorphism leads to a permutation of ${m,n,p}$ and so we assume without loss of generality that $m \leq n \leq p$. 
\end{rem}
Our first aim is to use Theorem \ref{thm:sophie} to construct a list of all possible triangle-point groups. We do this by classifying the normal subgroups, and the corresponding quotients, of the groups $G_1, \ldots, G_{11}$. However, some smaller examples will appear as quotients of many of the above groups. Thus, to significantly reduce the number of normal subgroups that we must classify, we first consider small examples of triangle-point groups.

\begin{prop}
\label{prop:tp12}
Suppose that $G$ is a triangle-point group of order at most 12. Then $G$ is either a dihedral group or an elementary abelian group of order 8. 
\end{prop}
\begin{proof}
Suppose that $G = \langle a,b,c \rangle$ is a triangle-point group. Then $G$ contains the subgroup $\langle a,b \rangle \cong 2^2$ and so the order of $G$ must be a multiple of 4 and so must be equal to $4$, $8$ or $12$. Up to isomorphism, the only groups of these orders that are generated by their involutions are $D_4$, $D_8$, $2^3$ and $D_{12}$ and it is easy to check that each of these is indeed a triangle-point group.  
\end{proof}

In Tables \ref{tab:normals} and \ref{tab:normals11} below, for each $i \in [1..11]$ we give a complete list of the non-trivial normal subgroups $N \triangleleft \, G_i$ such that $[G_i : N] > 12$. Note that the groups $G_3$ and $G_5$ have no such normal subgroups and are thus omitted from these tables.

The lists of normal subgroups in Table \ref{tab:normals} have been calculated in GAP \cite{GAP} using explicit generators of each of the groups $G_1, \ldots, G_{10}$ and by using the group presentation in the case of $G_{11}$. In most cases, the generators used below are given in \cite{Norton85}.  In particular, where possible, we choose these generators to be elements of $A_{12}$. 
\begin{prop}
For $1 \leq i \leq 10$, Table \ref{tab:normals} gives
\begin{itemize}
\item elements $a,b,c \in G_i$ such that $G_i = \langle a,b,c \rangle$ as a triangle-point group;
\item generators in terms of $a,b,c$ of all normal subgroups $N \trianglelefteq G_i$ such that $[G_i:N] > 12$;
\item the isomorphism types for the corresponding quotients $G_i/N$.
\end{itemize}
Table \ref{tab:normals11} gives generators of all normal subgroups $N \trianglelefteq G_{11}$ such that $[G_{11}:N] > 12$ and the isomorphism types for the corresponding quotients $G_{11}/N$.
\end{prop}
\begin{table}%
\begin{center}
\begin{tabular}{|>{$}l<{$} | >{$}l<{$} >{$}l<{$} >{$}l<{$} >{$}l<{$} |}
\hline
&&&& \\
G & \textrm{Generators } a,b,c & N & X \subset G \textrm{ s.t. } N = \langle X \rangle^G & G/N \\ 
&&&& \\ \hline
&&&& \\
G_1 & \ml{l}{(1,2)(3,4)\\(1,3)(2,4)(5,6)(7,8) \\ (1,5)(2,7)} 
    & \ml{l}{ 2^2 \\ 2^2 \\ 2^2 \\ 2} 
    & \ml{l}{ (ac)^2 \\ (bc)^2 \\ (abc)^2 \\ (a \cdot b^c)^2}
    & \ml{c}{ 2 \times D_8 \\ 2 \times D_8 \\ 2 \times D_8 \\ 2^4:2} \\
&&&& \\
G_2 & \ml{l}{(1,2)(3,4) \\ (5,6)(7,8) \\ (1,2)(3,9)(4,5)(6,10) } 
    & \ml{l}{3^2 \\ 2} 
    & \ml{l}{(a \cdot b^c)^2 \\ (a \cdot b^c)^3} 
    & \ml{l}{2 \times D_8 \\ (S_3 \times S_3):2} \\
&&&& \\
G_4 & \ml{l}{(1,2)(3,4) \\ (1,2)(3,4)(5,6)(7,8) \\ (1,9)(2,5)(3,4)(7,8)}
    & 2 
    
    & (ac)^2 
    & S_5 \\
&&&& \\
G_6 & \ml{l}{(1,2)(3,4) \\ (1,3)(2,4)(5,6)(7,8)(9,10)(11,12) \\ (1,2)(3,5)(4,7)(6,9)(8,11)(10,12) }
    & \ml{l}{2^4 \\ 2^4 \\ 2^3 \\ 2^3 \\ 2^3 \\ 2^2 \\ 2 }
    & \ml{l}{(bc)^3, (abc)^3 \\
            (ac)^2 \\
            (ab \cdot b^c)^3, (a^c \cdot c^b)^2 \\
            (bc)^3 \\
            (abc)^3 \\
            (a^c \cdot c^b)^2 \\
            (ab \cdot b^c)^3 }
    & \ml{l}{S_4 \\ 2^2 \times S_3 \\ 2 \times S_4 \\ 2 \times S_4 \\ 2 \times S_4 \\ 2^2 \times S_4 \\ 2^4:D_{12}} \\
&&&& \\  
G_7 & \ml{l}{   (1,3)(2,15)(4,13)(6,12)(7,11)(14,16) \\  
                (1,11)(2,12)(3,9)(4,10)(5,6)(13,14) \\
                (1,3)(2,15)(4,13)(6,12)(7,11)(14,16)}
    & 2^4 & (ac)^3 &  A_5 \\
&&&& \\
G_8 & \ml{l}{   (1,2)(7,8) \\
                (1,2)(3,4)(5,6)(9,10) \\
                (1,3)(4,5)(7,8)(9,10)}
    & 2 & ((bc)^3 \cdot b^{ca})^3 & S_6 \\
&&&& \\
G_9 & \ml{l}{   (1,2)(3,4)(5,6)(7,8) \\
                (1,8)(2,7)(3,4)(5,6) \\
                (2,5)(3,6)(9,10)(11,12)}
    & \ml{l}{2^4:3 \\ 2^4:3 \\ 2^5 \\ 2^4 \\ 2} 
    & \ml{l}{   (ac)^2, (a \cdot b^c)^2 \\ 
                (bc)^2, (a \cdot b^c)^2 \\ 
                (abc)^2, (a \cdot b^c)^2 \\ 
                (a \cdot b^c)^2 \\ 
                a \cdot (b \cdot c^{ac})^3} 
    & \ml{l}{2^2 \times S_3 \\ 2^2 \times S_3 \\ S_3 \times S_3 \\ 2 \times S_3 \times S_3 \\ 2^4 :(S_3 \times S_3) } \\
&&&& \\
G_{10}  & \ml{l}{   (1,2)(3,4)(5,6)(7,8)(9,10)(11,12) \\
                    (1,3)(2,4)(5,8)(6,7)(9,12)(10,11) \\
                    (1,7)(2,6)(3,9)(4,11)(5,10)(8,12)}
        & \ml{l}{2^5 \\ 2} 
        & \ml{l}{((ac)^2(bc)^2)^2 \\ (c^a \cdot (bc)^3)^3}
        & \ml{l}{S_5 \\ 2^4 : S_5} \\ 
&&&& \\ \hline
\end{tabular}
\caption{Some normal subgroups of $G_1, \ldots, G_{10}$}
\label{tab:normals}
\end{center}
\end{table}
\begin{table}%
\begin{center}
\begin{tabular}{ |>{$}l<{$} >{$}l<{$} >{$}l<{$} |} 
        
\hline
&& \\
N & X \subset G_{11} \textrm{ s.t. } N = \langle X \rangle^{G_{11}}& G_{11}/N \\ 
&& \\ \hline
&& \\
\ml{l}{ ((3^4:3):3) \\
        ((3^4:3):3) \\
        ((3^4:3):3) \\
        (3^4:3):2 \\
        (3^4:3):2 \\
        (3^4:3):2 \\
        3^4:3 \\
        3^4:3 \\
        3^4:3 \\
        3^4:2 \\        
        3^4:2 \\        
        3^4:2 \\        
        3^4 \\
        3^4 \\
        3^4 \\
        3^4 \\
        3^3 \\
        3^3 \\
        3^3 \\
        3^2 \\
        3^2 \\
        3^2 \\
        3} &
\ml{l}{ (ac)^2 \\
        (bc)^2, (a \cdot b^c)^2 \\
        (abc)^2 \\
        (ac)^3, (ab \cdot b^c)^2 \\
        (bc)^3, (ab \cdot a^c)^2 \\
        (abc)^3, (a \cdot a^c)^2 \\
        (a \cdot b^c)^2 \\
        (ab \cdot b^c)^2 \\
        (ab \cdot a^c)^2 \\
        (ac)^3 \\
        (bc)^3 \\
        (abc)^3 \\
        (acbcacb)^2 \\
        (a \cdot c^{bc})^2 \\
        (b \cdot c^{ac})^2 \\
        (ab \cdot c^{ac})^2 \\
        (ab \cdot a^{cbc})^2, (a \cdot b^{cabc})^2 \\
        (ab \cdot b^{cac})^2, (a \cdot b^{cabc})^2 \\
        (ab \cdot b^{cac})^2, (ab \cdot a^{cbc})^2 \\
        (a \cdot b^{cabc})^2, \\
        (ab \cdot b^{cac})^2, \\
        (ab \cdot a^{cbc})^2 \\
        c^{acbcacb}\cdot c^{bcacbca}} &
        
\ml{l}{ 2^2 \times S_3 \\
        2^2 \times S_3 \\
        2^2 \times S_3 \\
        S_3 \times S_3 \\ 
        S_3 \times S_3 \\ 
        S_3 \times S_3 \\ 
        2 \times S_3 \times S_3 \\
        2 \times S_3 \times S_3 \\
        2 \times S_3 \times S_3 \\
        (3^{1+2}_+ : 2^2) \\
        (3^{1+2}_+ : 2^2) \\
        (3^{1+2}_+ : 2^2) \\
        S_3 \times S_3 \times S_3 \\
        2 \times (3^{1+2}_+ : 2^2) \\
        2 \times (3^{1+2}_+ : 2^2) \\
        2 \times (3^{1+2}_+ : 2^2) \\
        S_3 : (3^{1+2}_+ : 2^2) \\
        S_3 : (3^{1+2}_+ : 2^2) \\
        S_3 : (3^{1+2}_+ : 2^2) \\
        (3^2 : 2) : (3^{1+2}_+ : 2^2) \\
        (3^2 : 2) : (3^{1+2}_+ : 2^2) \\
        (3^2 : 2) : (3^{1+2}_+ : 2^2) \\
        (3^3 : 2) : (3^{1+2}_+ : 2^2) } \\ \hline
\end{tabular}
\caption{Some normal subgroups of $G_{11}$}
\label{tab:normals11}
\end{center}
\end{table}
\section{The main result}
\label{sec:main}
In this section we prove our main result concerning Majorana representations of triangle-point groups.
\begin{thm}
\label{thm:main}
Suppose that $V$ is a Majorana algebra generated by three Majorana axes $a_0,a_1,a_2$ such that the dihedral algebra $\langle \langle a_0,a_1 \rangle \rangle$ is of type $2A$. Then $V$ must occur as a Majorana representation of one of $27$ groups, each of which occurs as a subgroup of the Monster.  
\end{thm}
By considering the list of triangle-point groups of order less than 12 in Proposition \ref{prop:tp12} and the list of larger triangle-point groups in Tables \ref{tab:normals} and \ref{tab:normals11}, we see that, up to isomorphism, there are a total of 37 triangle-point groups. By comparing these with Norton's list of triangle-point subgroups of the Monster (Table 3 of \cite{Norton85}), we see that there are ten such groups that do not $2A$-embed into the Monster, details of which are given in Table \ref{tab:badtp}.

In the remainder of this section, we consider these ten groups and show that none of them can admit a Majorana representation of the form $(G,T,V)$ where $G = \langle a,b,c \rangle$ and $a,b,c,ab \in T$. In doing so, we also show that these groups cannot exist as triangle-point subgroups of the Monster (as otherwise they would have to admit such a Majorana representation) and so prove that Norton's list of triangle-point groups is complete.

\begin{table}%
\begin{center}
\begin{tabular}{|>{$}l<{$} >{$}l<{$} >{$}c<{$} |} \hline
 G & |G| & \textrm{\begin{tabular}{@{}c@{}}Quotient of \\ $G_i$ for $i$ in \end{tabular}  } \\ \hline
 2 \times S_3 \times S_3 & 72 & 9,11 \\
 S_6 & 720 & 8 \\
 (2^4:(S_3 \times S_3)) \times 2 & 1152 & 9 \\
 2^4:S_5 & 1920 & 10 \\
 S_3 \times S_3 \times S_3 & 216 & 11 \\
 2 \times (3^{1+2}_+:2^2) & 216 & 11 \\
 (3:2) : (3^{1+2}_+:2^2) & 648 & 11 \\
 (3^2:2) : (3^{1+2}_+:2^2) & 1944 & 11 \\
 (3^3:2) : (3^{1+2}_+:2^2) & 5832 & 11 \\
 (3^4:2) : (3^{1+2}_+:2^2) & 17496 & 11 \\ \hline
\end{tabular}
\caption{The triangle-point groups that do not $2A$-embed into the Monster}
\label{tab:badtp}
\end{center}
\end{table}
Throughout this section, we make use of the following result, which is Lemma 8.6.3 in \cite{Ivanov09}.
\begin{lem}
\label{lem:elab}
Suppose that there exists a group $G$ that admits a Majorana representation $(G,T,V)$. Suppose also that $G$ contains a subgroup $K$  isomorphic to the elementary abelian group of order 8. Then there must exist at least one non-identity element of $K$ that does not lie in $T$.
\end{lem}
\begin{proof}
Suppose that $K:=\langle t_0,t_1,t_2 \rangle$ and suppose for contradiction that all non-identity elements of $K$ lie in $T$. Since any two axes $\psi(t_i)$ and $\psi(t_j)$ generate a $2A$ algebra, 
\[
\psi(t_1) - \psi(t_0t_1), \,  \psi(t_2) - \psi(t_0t_2) \textrm{ and } \psi(t_1t_2) - \psi(t_0t_1t_2)
\]
are all $\frac{1}{2^2}$-eigenvectors of $\psi(t_0)$. However,
\[
(\psi(t_1) - \psi(t_0t_1)) \cdot (\psi(t_2) - \psi(t_0t_2)) = -\frac{1}{2^2}(\psi(t_1t_2) - \psi(t_0t_1t_2)).  
\]
is also a $\frac{1}{2^2}$-eigenvector of $\psi(t_0)$. This contradicts the fusion rules and so such a representation cannot exist. 
\end{proof}
In most of the cases below, we have explicit generators for the groups in question. However, as Tables \ref{tab:normals} and \ref{tab:normals11} provide an exhaustive list of all triangle-point groups of order greater than $12$, we can also use these to determine the exact presentations of these groups. Whether we use explicit generators or the group presentation for our calculations is simply a question of clarity.
\subsection{The group $2 \times S_3 \times S_3$}
\begin{prop}
Suppose that $G = \langle a,b,c \rangle \cong 2 \times S_3 \times S_3$ is a triangle-point group and suppose that $T \subseteq G$ such that $a,b,c,ab \in T$, then there exist no Majorana representations of the form $(G,T, V)$.
\end{prop}
\begin{proof}
From Tables \ref{tab:normals} and \ref{tab:normals11}, we see that $G$ occurs either as a quotient of $G_9$, or as a quotient of $G_{11}$. In either case, we must have 
\[
G = \langle a,b,c \mid a^2, b^2, c^2, (ab)^2, (ac)^6, (bc)^6, (abc)^6, (a\cdot b^c)^{r_1}, (ab \cdot b^c)^{r_2}, (ab \cdot a^c)^{r_3} \rangle
\]
where $(r_1,r_2,r_3) \in \{(2,6,6),(6,2,6),(6,6,2)\}$.
We first suppose that $(r_1,r_2,r_3) = (2,6,6)$ and show that the group 
\[
K := \langle a, b, (abc)^3 \rangle \leq G
\]
is isomorphic to $2^3$ and that all of its non-identity elements are contained in $T$. Using the presentation of $G$ in GAP, we have checked that $o(ac) = o(bc) = o(abc) = 6$. By assumption, $a,b,ab \in T$ and, from axiom (M8), as $o(abc) = 6$, $(abc)^3 \in T$. Using the presentation of $G$, and in particular the relation $(a \cdot b^c)^2 = 1$, we can show that 
\begin{align*}
a \cdot (abc)^3 &= ((bc)^3)^{ac} \in T \\
b \cdot (abc)^3 &= ((ac)^3)^{bc} \in T \\
ab \cdot (abc)^3 &= c^{abc} \in T 
\end{align*}
and so $2^3 \cong K \subset T \cup \{e\}$. This is a contradiction with Lemma \ref{lem:elab} and so no such representation can exist. 
In the case that  $(r_1,r_2,r_3) = (6,2,6)$ or $(6,6,2)$, we take $K$ to be $\langle a,b, (ac)^3 \rangle$ or $\langle a,b,(bc)^3 \rangle$ respectively. In either case, we find that $2^3 \cong K \subset T \cup \{e\}$, again giving a contradiction.
\end{proof}
\subsection{The group $S_6$}
\begin{prop}
Suppose that $G = \langle a,b,c \rangle \cong S_6$ is a triangle-point group and suppose that $T \subseteq G$ such that $a,b,c,ab \in T$, then there exist no Majorana representations of the form $(G,T, V)$.
\end{prop}
\begin{proof}
From Tables \ref{tab:normals} and \ref{tab:normals11}, we see that $G$ must occur as the group
\[
\langle a,b,c \mid a^2, b^2, c^2, (ab)^2, (ac)^6, (bc)^6, (abc)^5, (a\cdot b^c)^4, x^3 \rangle
\]
where $x = (bc)^3 \cdot b^{ca}$, so that $x^3$ is the central element of $G_8$. In this case, we use explicit generators. If we pick
\begin{align*}
a &:= (1,2)(3,4)(5,6)\\
b &:= (5,6) \\
c &:= (2,3)(4,5)
\end{align*}
then $a,b,c$ satisfy the relations above and generate the group $S_6$. Thus we may take $G = \langle a,b,c \rangle$. By definition, $T$ must contain the conjugacy classes
\[
a^G = (1,2)(3,4)^G, \, b^G = (5,6)^G, \, (ab)^G = (1,2)(3,4)(5,6)^G.
\]
In particular, as the conjugacy classes in $S_6$ are indexed by the cycle types of their elements, this must mean that \emph{all} involutions of $G$ are contained in $T$. 
Finally, $G$ contains the subgroup $\langle (1,2),(3,4),(5,6) \rangle \cong 2^3$, all of whose non-identity elements must be contained in $T$. This is in contradiction with Lemma \ref{lem:elab} and the result follows. 
\end{proof}
\subsection{The group $(2^4:(S_3 \times S_3))\times 2$}
\begin{prop}
Suppose that $G = \langle a,b,c \rangle \cong  (2^4:(S_3 \times S_3)) \times 2$ is a triangle-point group and suppose that $T \subseteq G$ such that $a,b,c,ab \in T$, then there exist no Majorana representations of the form $(G,T, V)$.
\end{prop}
\begin{proof}
In this case, we must have $G = G_9$. Although we have explicit generators of the group, it is easier to consider $G$ as a finitely presented group with generators $a,b,c$. 
We now let $x := ab \cdot c^{ac}$ and claim that 
\[
K := \langle a,b,x^3 \rangle
\]
is isomorphic to $2^3$ and that all the non-identity elements of $K$ lie in $T$. By definition, $a,b,ab \in T$. We calculate that
\begin{align*}
a \cdot x^3 &= (b \cdot c^{ac})^3   \\
b \cdot x^3 &=((ac)^3)^{bcacac} \\
ab \cdot x^3 &=c^{acabcacac}.
\end{align*}
Either by using explicit generators, or by calculating with the group presentation in GAP, we see that $x$, $b \cdot c^{ac}$ and $ac$ are all of order 6. Moreover, as they are each the product of two elements of $T$, by axiom (M8), their cubes must all also lie in $T$. This shows that $K$ must be isomorphic to $2^3$ and that all non-identity elements of $K$ are contained in $T$. This is in contradiction with Lemma \ref{lem:elab} and the result follows. 
\end{proof}
\subsection{The group $2^4 : S_5$}
We deal with this group using slightly different techniques to the other cases. We begin by noting that from Tables \ref{tab:normals} and \ref{tab:normals11}, we see that $G$ must occur as the group
\[
\langle a,b,c \mid a^2, b^2, c^2, (ab)^2, (ac)^6, (bc)^6, (abc)^6, (a\cdot b^c)^5, (ab \cdot b^c)^5, (ab \cdot a^c)^5, (c \cdot b^{ca})^4, x^3 \rangle
\]
where $x = c^a \cdot (bc)^3$, so that $x^3$ is the central element of $G_{10}$.
If we take
\begin{align*}
a &:= (1,2)(3,4)(5,6)(7,8)(9,10)(11,12) \\
b &:= (1,3)(2,4)(5,6)(7,8)(13,14)(15,16) \\
c &:= (1,12)(3,14)(4,6)(5,16)(7,11)(9,13).
\end{align*}
then $a,b,c$ generate a group of order $1920$ and satisfy the presentation of $G$ and so we may take $G = \langle a,b,c \rangle$.  

We will show that $G$ contains a subgroup $K \cong 2 \times D_8$ and that there exist no representations of the form $(K,K\cap T, U)$. This will in turn show that there exist no  representations of the form $(G,T,V)$.
\begin{lem}
\label{lem:2D8}
Let $K := \langle (1,2), (1,3)(2,4), (5,6) \rangle \cong 2 \times D_8$ then $K$ contains eleven involutions, which we label $t_i$ for $1 \leq i \leq 11$ as below. 
\begin{center}
\begin{tabular}{| >{$} c <{$} >{$} l <{$} | >{$} c <{$} >{$} l <{$} | >{$} c <{$} >{$} l <{$} |} \hline
i & t_i & i & t_i & i & t_i \\ \hline
1 & (1,2)       & 5 & (1,3)(2,4)(5,6)   & 9  & (1,2)(3,4) \\
2 & (3,4)       & 6 & (1,4)(2,3)(5,6)   & 10 & (1,2)(3,4)(5,6) \\
3 & (1,2)(5,6)  & 7 & (1,3)(2,4)        & 11 & (5,6) \\
4 & (3,4)(5,6)  & 8 & (1,4)(2,3)        & & \\ \hline
\end{tabular}
\end{center}
If we let $S := \{t_1, \ldots, t_{10} \}$ then there exist no Majorana representations of the form $(K,S,U)$. 
\end{lem} 
\begin{proof}
We suppose for contradiction that such a representation exists and show that it cannot obey axiom M1. In the following, we let $a_i := \psi(t_i)$ for $1 \leq i \leq 10$. Note that $t_1t_4 = (1,2)(3,4)(5,6) = t_{10}$ and so the algebra $\langle \langle a_1, a_4 \rangle \rangle$ is of type $2A$ and 
\[
a_1 \cdot a_4 = \frac{1}{2^3}(a_1 + a_4 - a_{10}).
\]
Now $o(t_1t_5) = o(t_4t_5) = 4$ and $(t_1t_5)^2, (t_4t_5)^2 \in S$ and so the algebras $\langle \langle a_1, a_5 \rangle \rangle $ and $\langle \langle a_4, a_5 \rangle \rangle$ are of type $4B$ and 
\[
(a_1, a_5) = (a_4, a_5) = \frac{1}{2^6}. 
\]
Finally, $t_{10}t_5 = (1,4)(2,3) = t_8$ and so $\langle \langle a_{10}, a_5 \rangle \rangle $ is of type $2A$, $(a_{10},a_5) = \frac{1}{2^3}$ and 
\[ 
(a_1 \cdot a_4, a_5) = -\frac{3}{2^8}.
\]
Similarly, we calculate that
\[
a_4 \cdot a_5 = \frac{1}{2^6}(a_4 + a_5 - a_3 - a_6 + a_9)
\]
and that
\[
(a_1, a_4 \cdot a_5) = \frac{1}{2^6} \neq (a_1 \cdot a_4, a_5)
\]
which is in contradiction with axiom M1, showing that such an algebra cannot exist. 
\end{proof}
\begin{prop}
Suppose that $G = \langle a,b,c \rangle \cong 2^4:S_5$ is a triangle-point group and suppose that $T \subseteq G$ such that $a,b,c,ab \in T$, then there exist no Majorana representations of the form $(G,T,V)$.
\end{prop}
\begin{proof}
If we let 
\[
x := a, \, y := b^{cacac} \textrm{ and } z := ((ac)^3)^b
\]
then it is easy to check that the map $f$ that sends
\begin{align*}
x & \mapsto (1,2)(3,4)(5,6) \\
y & \mapsto (3,4) \\
z & \mapsto (1,4)(2,3) 
\end{align*}
is an isomorphism from $ K := \langle x,y,z \rangle$ to $2 \times D_8$. 

By definition and by axiom (M8), we have $x,y,z \in T$. Moreover,
\begin{align*}
xy &= (ab)^{cacac} \in T,\\ 
xz &= c^{acb} \in T, \\
(yz)^2 &= (b \cdot b^{cacac})^3 \in T. 
\end{align*}
By considering the conjugacy classes of $2 \times D_8$, we see that 
\[
|x^K \cup y^K \cup z^K \cup (xy)^K \cup (xz)^K \cup ((yz)^2)^K| = 10
\]
and so, as $K$ contains $11$ involutions in total, $|K \cap T|$ is equal to $10$ or $11$. 

If $|K \cap T| = 11$ then $K$ contains an elementary abelian subgroup of order $8$ all of whose involutions are contained in $T$ and so the representation $(K, K \cap T, U)$ cannot exist. If $|K \cap T| = 10$ then $(K, K \cap T, U)$ is the representation in Lemma \ref{lem:2D8} and so equally cannot exist. Thus we may conclude that there exist no representations of the form $(G,T,V)$, as required. 
\end{proof}
\subsection{The group $S_3 \times S_3 \times S_3$}
\begin{prop}
Suppose that $G = \langle a,b,c \rangle \cong S_3 \times S_3 \times S_3$ is a triangle-point group and suppose that $T \subseteq G$ such that $a,b,c,ab \in T$, then there exist no Majorana representations of the form $(G,T, V)$.
\end{prop}
\begin{proof}
From Tables \ref{tab:normals} and \ref{tab:normals11}, we see that $G$ must occur as the group
\[
\langle a,b,c \mid a^2, b^2, c^2, (ab)^2, (ac)^6, (bc)^6, (abc)^6, (a\cdot b^c)^6, (ab \cdot b^c)^6, (ab \cdot a^c)^6, (c^a \cdot c^{bc})^3, x^2 \rangle
\]
where $ x = acbcacb$. In this case, we use explicit generators. If we pick
\begin{align*}
a &:= (1,2)(4,5)\\
b &:= (4,5)(7,8) \\
c &:= (1,3)(4,6)(7,9)
\end{align*}
then $a,b,c$ satisfy the relations above and generate the group $S_3 \times S_3 \times S_3$. Thus we may take $G = \langle a,b,c \rangle$. By assumption, $T$ must contain the conjugacy classes $a^G,b^G,c^G,(ab)^G$. By axiom (M8), it must also contain
\begin{align*}
(ac)^3 &= (7,9) \\
(bc)^3 &= (1,3) \\
(abc)^3 &= (4,6). 
\end{align*}
We now let 
\[
K := \langle (1,3),(4,6),(7,9) \rangle \leq G.
\]
then $K$ is clearly elementary abelian of order 8 and, from the above discussion, all the non-identity elements of $K$ are contained in $T$. This is a contradiction with Lemma \ref{lem:elab} and so such a representation cannot exist. 
\end{proof}
\subsection{The group $2 \times (3^{1+2}_+:2^2)$}
\begin{prop}
Suppose that $G = \langle a,b,c \rangle \cong 2 \times (3^{1+2}_+:2^2)$ is a triangle-point group and suppose that $T \subseteq G$ such that $a,b,c,ab \in T$, then there exist no Majorana representations of the form $(G,T, V)$.
\end{prop}
\begin{proof}
From Tables \ref{tab:normals} and \ref{tab:normals11}, we see that $G$ must occur as the group
\[
\langle a,b,c \mid, a^2, b^2, c^2, (ab)^2, (ac)^6, (bc)^6, (abc)^6, (a\cdot b^c)^6, (ab \cdot b^c)^6, (ab \cdot a^c)^6, (c^a \cdot c^{bc})^3, y^2 \rangle
\]
where $y \in \{a \cdot c^{bc}, b \cdot c^{ac}, ab \cdot c^{ac}\}$. 

Note that any of the possible values for $y$ can be sent to any other by a suitable permutation of the generators $a,b,ab$. Moreover, such a permutation preserves all other relations in the presentation of these groups (to show this for the relation $(c^a \cdot c^{bc})^3 = 1$ requires some calculation, in all other cases it is clear). Thus, permutating $a$, $b$ and $ab$ induces pairwise isomorphisms between the three groups arising from the different choices of $y$. 

Without loss of generality, we can now pick $y = a \cdot c^{bc}$ and let
\[
a = (1,4)(2,6)(3,5)(8,9), b = (1,4)(2,8)(6,9)(10,11), c = (2,7)(3,4)(5,9).
\]
Then it is easy to check that $a,b,c$ satisfy the presentation of $G$ and generate a group of order $216$. Thus we may take $G = \langle a,b,c \rangle$. We then calculate that
\begin{align*}
ac &= (1,3,9,8,5,4)(2,6,7) \\ 
b \cdot c^{ac} &= (1,8,6)(2,4,9)(10,11) \\
ab \cdot c^{ac} &= (1,9,6,4,8,2)(3,5)(10,11)\\
\end{align*} 
are all of order 6. We now let
\[
K := \langle a,b, y \rangle.
\]
where $y := (b \cdot c^{ac})^3$. From axiom (M8), $(b \cdot c^{ac})^3, (ab \cdot c^{ac})^3, (ac)^3 \in T$. Thus 
\begin{align*}
a \cdot y &= (ab \cdot c^{ac})^3 \in T \\
b \cdot y &= ((ac)^3)^{b(ac)^3} \in T \\
ab \cdot y &= c^{acabc(ac)^2} \in T. 
\end{align*}
We now have $K \cong 2^3 \subseteq T \cup \{e\}$, which is a contradiction with Lemma \ref{lem:elab} and so such a representation cannot exist. 
\end{proof}
\subsection{The remaining quotients of the group $G_{11}$}
Here we consider the case where $G = \langle a,b,c \rangle \cong (3^i:2):(3^{1+2}_+:2^2)$ for $i=1,2,3,4$. The following is Proposition 3.52 in \cite{Decelle13}.
\begin{lem}
\label{lem:r4}
Let $K$ be the quotient of $G^{(6,6,6)}$ with the additional relations 
\[
(a\cdot b^c)^6 = (ab \cdot b^c)^6 = (ab \cdot a^c)^6 = (c \cdot b^{ca})^{r_4} = 1
\]
then 
\begin{itemize}
	\item if $r_4 \in \{1,5\}$ then $K \cong D_{12}$;
	\item if $r_4 \in \{3\}$ then $K \cong G^{(3,6,6)} \cong 3^{1+2}_+ : 2^2$;
	\item if $r_4 \in \{2,4\}$ then $K \cong 2 \times G^{(3,6,6)} \cong 2 \times (3^{1+2}_+ : 2^2)$. 
\end{itemize}
\end{lem}
We will also require the following result.
\begin{lem}
\label{lem:m66}
Let $G := G^{(m,6,6)}$ then 
\begin{itemize}
\item if $m = 1$, $G \cong 2^2$; 
\item if $m = 2$, $G \cong 2 \times D_{12}$; 
\item if $m = 3$, $G \cong 3^{1+2}_+:2^2$.
\end{itemize}
In particular, if $m \in \{1,2,3\}$ then $|G^{(m,6,6)}| \leq 108$. 
\end{lem}
\begin{proof}
We let $G := G^{(m,6,6)}$ and deal with the cases $m = 1,2,3$ in turn.
\begin{itemize}
\item $m=1$: Here $a = c$ and so $G = \langle a,b \mid a^2, b^2, (ab)^2 \rangle \cong 2^2$.
\item $m=2$: Here $a$ commutes with $b$ and $c$ and is of order $2$ and so $G = 2 \times \langle b,c \rangle \cong 2 \times D_{12}$.
\item $m=3$: This case is given in Lemma \ref{lem:r4} above. \qedhere
\end{itemize} 
\end{proof}

\begin{prop}
Suppose that $G = \langle a,b,c \rangle \cong (3^i : 2) : (3^{1+2}_+:2^2)$ for $i = \{1,2,3,4\}$ is a triangle-point group and suppose that $T \subseteq G$ such that $a,b,c,ab \in T$, then there exist no Majorana representations of the form $(G,T, V)$.
\end{prop}
\begin{proof}
Let $m := o(ac)$, $n := o(bc)$ and $p := o(abc)$. We will show that we must have $(m,n,p) = (6,6,6)$. Suppose for contradiction that this is not true. Then $G$ must be isomorphic to a quotient of $G^{(m,6,6)}$ for $m \in \{1,2,3\}$. However, $|G| \geq 648$, in contradiction with with Lemma \ref{lem:m66} above, and so we must have $(m,n,p) = (6,6,6)$. With axiom (M8), this implies that $(ac)^3, (bc)^3, (abc)^3 \in T$.

We now consider the element $x := b \cdot (ac)^3$. As $G$ is a triangle-point group and $(ac)^3 \in T$, we must have $o(x) \leq 6$. If we were to have $o(x) < 6$ then $R_4 = x^{cac}$ would also be of order strictly less than $6$ and so $G$ would have to exist as the quotient of one of the groups in Lemma \ref{lem:r4}. Comparison of orders again shows that this cannot be the case, and so we get $o(R_4) = o(x) = 6$. 

We claim that 
\[
K := \langle a, b, x^3 \rangle 
\]
is elementary abelian of order $8$ and that all its non-identity elements lie in $T$. By assumption, $a,b,ab \in T$ and, by axiom (M8), $x^3, (abc)^3, (ac)^3 \in T$. 
\begin{align*}
a \cdot x^3 &= ((abc)^3)^{bcabcabca} \\
b \cdot x^3 &= c^{acbcacac} \\
ab \cdot x^3 &= (ac)^3.
\end{align*}
We now have $K \cong 2^3 \subseteq T \cup \{e\}$, which is a contradiction with Lemma \ref{lem:elab} and so such a representation cannot exist. 
\end{proof}

\end{document}